\documentclass[12pt]{amsart}
\usepackage{eurosym}
\usepackage{amsmath,amssymb,amsfonts,amsthm,amscd,indentfirst}
\usepackage{amsmath,amssymb,amsfonts,amsthm,amscd}
\usepackage{indentfirst}
\usepackage{hyperref}
\usepackage{graphicx}

\numberwithin{equation}{section}

\newtheorem{prop}{Proposition}[section]
\newtheorem{theo}[prop]{Theorem}
\newtheorem{lemm}[prop]{Lemma}

\newtheorem{defi}[prop]{Definition}

\def\and{\quad{\rm and}\quad}

\def\<{\langle}
\def\>{\rangle}

\begin{document}
\title[Curvature estimates for $n-1$ Hessian equation]{A simple proof of
curvature estimates for the $n-1$ Hessian equation}
\author[Siyuan Lu and Yi-Lin Tsai]{Siyuan Lu and Yi-Lin Tsai}
\address{Department of Mathematics and Statistics, McMaster University, 1280
Main Street West, Hamilton, ON, L8S 4K1, Canada.}
\email{siyuan.lu@mcmaster.ca}
\email{tsaiy11@mcmaster.ca}
\thanks{Research of the first author was supported in part by NSERC
Discovery Grant.}

\begin{abstract}
In [Amer. J. Math. 141 (2019), no. 5, 1281-1315], Ren and Wang proved the curvature estimates for the $n-1$ curvature equation. The purpose of this note is to give a simple proof of their theorem.
\end{abstract}

\maketitle

\section{Introduction}

In \cite{RW1}, Ren and Wang established a global curvature estimate for a
closed $k$-convex hypersurface $M$ in $\mathbb{R}^{n+1}$ satisfying the following equation for $k=n-1$,
\begin{equation}
\sigma _{k}(\kappa (X))=f(X,\nu (X)),\quad\forall X\in M,
\label{Equation-Sigma_k}
\end{equation}%
where $\sigma _{k}$ is the $k$-th elementary symmetric function, $\nu $ and $%
\kappa =(\kappa _{1},\cdots ,\kappa _{n})$ are the unit outer normal and
principal curvatures of the hypersurface $M$. In particular, $\sigma_k(\kappa)$ corresponds to the mean curvature,
scalar curvature and Gauss curvature when $k=1,2$ and $n$.

Equation (\ref{Equation-Sigma_k}) arises naturally from the study of geometric problems. For example, the Minkowski problem \cite{CY,Nirenberg,
Pog53,Pogb}, the prescribed Weingarten curvature problem \cite{A3,GG}, the
prescribed curvature measure problem \cite{A2,GLL,GLM,Pog53}, and the
prescribed curvature problem \cite{BK,CNS86,TW}.

Equation (\ref{Equation-Sigma_k}) is a fully nonlinear equation. In the study of fully nonlinear equations, curvature estimates play an essential role. We now briefly mention the history of curvature estimates for Equation (\ref{Equation-Sigma_k}). When $k=1$, it is quasi-linear, curvature estimate follows from the standard theory of quasi-linear equations. When $k=n$, it is of Monge-Amp\`{e}re type, curvature estimate was established by Caffarelli, Nirenberg and Spruck \cite{CNS84}. When $f$ is independent of $\nu $, curvature estimate was obtained by Caffarelli, Nirenberg and Spruck \cite
{CNS85} for a general class of fully nonlinear PDEs including $\sigma _{k}$
and $\frac{\sigma _{k}}{\sigma _{l}}$. If $f$ only depends on $\nu $, Guan and Guan \cite{GG} proved the curvature estimate. With extra conditions on the dependence of $f$ on $\nu $, Ivochkina \cite{I1,I2} established the $C^{2}$ estimate for Dirichlet problem of equation (\ref{Equation-Sigma_k}). For prescribed curvature measure problem $f(X,\nu )=\left\langle X,\nu\right\rangle \varphi (X)$, curvature estimate was obtained by Guan, Lin and Ma \cite{GLM} and Guan, Li and Li \cite{GLL}. For general $f(X,\nu )$, Guan, Ren and Wang \cite{GRW} solved the convex case for general $k$ (see a simpler proof by Chu \cite{Chu}), and they also established a curvature estimate for admissible solutions for $k=2$. Later, Spruck and Xiao \cite{SX} found an elegant proof for $k=2$, which works in the space form as well. For $k=n-1$ and $n-2$, curvature estimate was proved by Ren and Wang \cite{RW1,RW2}. We would like to remark that the general case $2<k<n-2$ is still open now.

In attempt to solve the general case, we would like to have a better understanding of the proof by Ren and Wang \cite{RW1} for $k=n-1$. Since their proof is quite sophisticated, it is desirable to have a simple proof of their theorem. On the other hand, a simpler and alternative proof may also guide us towards the resolution of the general case.

Before we state their theorem, let us recall the definition of Garding's $\Gamma _{k}$ cone
\begin{equation*}
\Gamma _{k}=\{\lambda \in \mathbb{R}^{n}:\sigma _{j}(\lambda )>0,1\leq j\leq k\}.
\end{equation*}

\begin{theo}\label{Theorem-General} [Ren and Wang \cite{RW1}]
Let $n\geq 3$, $k=n-1$ and let $M$ be a closed, strictly star-shaped hypersurface satisfying curvature equation (\ref{Equation-Sigma_k}) in $\mathbb{R}^{n+1}$ with $\kappa \in \Gamma _{n-1}$. Let $f\in C^{1,1}(\Gamma )$ be a positive function, where $\Gamma $ is an open neighborhood of the unit normal bundle of $M$ in $\mathbb{R}^{n+1}\times \mathbb{S}^{n}$. Then we have 
\begin{equation*}
\max_{X\in M;1\leq i\leq n}|\kappa _{i}(X)|\leq C,
\end{equation*}
where $C$ is a constant depending only on $n,\Vert X\Vert _{C^{0,1}},\inf f$ and $\Vert f\Vert _{C^{1,1}}$.
\end{theo}

The main idea of our new proof is to separate the arguments into two cases. If $M$ is semi-convex (see Definition \ref{semi-convex}), the main theorem holds by Lu's result \cite[Theorem 1.1]{LuSC}. The major difficulty is the case that $M$ is not semi-convex. By exploiting the structure of $\sigma_{n-1}$ and the assumption that $M$ is not semi-convex, we are able to establish a concavity inequality (see Lemma \ref{Lemma}) for admissible solutions. Our proof is elementary in nature and much simpler than Ren-Wang's \cite{RW1} concavity lemma. We believe our new idea can be used in other problems as well.

The organization of the note is as follows. In Section 2, we collect some formulas and lemmas for Hessian operator and the geometry of hypersurfaces. In Section 3, we establish the key concavity inequality. We will prove Theorem \ref{Theorem-General} in Section 4.

\section{Preliminaries}

In this section, we will collect some basic formulas and lemmas.

Let $\lambda=(\lambda_1,\cdots,\lambda_n)\in \mathbb{R}^n$, we will denote 
\begin{align*}
(\lambda|i)=(\lambda_1,\cdots,\lambda_{i-1},\lambda_{i+1},\cdots,\lambda_n)\in \mathbb{R}^{n-1},
\end{align*}
i.e. $(\lambda|i)$ is the vector obtained by deleting the $i$-th component
of the vector $\lambda$. Similarly, $(\lambda|ij)$ is the vector obtained by
deleting the $i$-th and $j$-th components of the vector $\lambda$.

We now collect some basic properties of Hessian operator, see for instance \cite[section 2]{RW2} and \cite{LT}.

\begin{lemm}
\label{Sigma_k-Lemma-0} For any $\lambda=(\lambda_1,\cdots,\lambda_n)\in \mathbb{R}^n$, we have 
\begin{align*}
\sigma_k(\lambda)=\lambda_i\sigma_{k-1}(\lambda|i)+\sigma_k(\lambda|i),\quad
\sum_i\sigma_{k}(\lambda|i)=(n-k)\sigma_{k}(\lambda),\quad
\sum_i\lambda_i\sigma_{k-1}(\lambda|i)=k\sigma_k(\lambda).
\end{align*}
\end{lemm}

\begin{lemm}
\label{Sigma_k-lemma} Let $\lambda=(\lambda_1,\cdots,\lambda_n)\in \Gamma_k$ with $\lambda_1\geq\cdots\geq\lambda_n$.

\begin{enumerate}
\item If $\lambda _{i}\leq 0$, then we have 
\begin{equation*}
-\lambda _{i}\leq \frac{(n-k)}{k}\lambda _{1}.
\end{equation*}

\item 
\begin{equation*}
\lambda _{1}\sigma _{k-1}(\lambda |1)\geq C(n,k)\sigma _{k}(\lambda ),
\end{equation*}
where $C(n,k)>0$ is a constant depending only on $n$ and $k$.

\item $\sigma _{j}(\lambda |i_{1}i_{2}...i_{s})>0$ for $\left\{i_{1},...,i_{s}\right\} \subset \left\{ 1,...,n\right\} $ provided that $j+s\leq k.$
\end{enumerate}
\end{lemm}

Let $\lambda(A)$ be the eigenvalue vector of a symmetric matrix $A=(a_{ij})$%
. Then we can define a function $F$ on the set of symmetric matrices by 
\begin{align*}
F(A)=f(\lambda(A)).
\end{align*}
Denote 
\begin{align*}
F^{pq}=\frac{\partial F}{\partial a_{pq}},\quad F^{pq,rs}=\frac{\partial^2F}{%
\partial a_{pq}\partial a_{rs}}.
\end{align*}

Suppose $A$ is diagonalized at $x_{0}$, then at $x_{0}$, we have 
\begin{equation*}
\sigma _{k}^{pq}(A)=\frac{\partial \sigma _{k}}{\partial \lambda _{p}}%
(\lambda )\delta _{pq}=\sigma _{k-1}(\lambda |p)\delta _{pq},
\end{equation*}%
\begin{equation*}
\sigma _{k}^{pq,rs}(A)=%
\begin{cases}
\frac{\partial ^{2}\sigma _{k}}{\partial \lambda _{p}\partial \lambda _{r}}%
(\lambda )=\sigma _{k-2}(\lambda |pr),\quad  & p=q,r=s,p\neq r, \\ 
-\frac{\partial ^{2}\sigma _{k}}{\partial \lambda _{p}\partial \lambda _{q}}%
(\lambda )=-\sigma _{k-2}(\lambda |pq), & p=s,q=r,p\neq q, \\ 
0, & \mathit{otherwise}.%
\end{cases}%
\end{equation*}

We now state a well-known formula, see for instance in \cite{LuCAG}.
\begin{lemm}
\label{support function} Let $u=\left\langle X,\nu\right\rangle$ be the support function, then we have 
\begin{align*}
u_i =&\ \sum_{k,l}g^{kl}h_{ik}\left\langle X,e_l\right\rangle, \\
u_{ij}=&\ \sum_{k,l}g^{kl}h_{ijk}\left\langle X,e_l\right\rangle+h_{ij}-\sum_{k,l}g^{kl}h_{ik}h_{jl}u,
\end{align*}
where $(g^{ij})$ is the inverse matrix of $(g_{ij})$.
\end{lemm}

For a fixed local orthonormal frame $(e_{1},\cdots ,e_{n})$, the Codazzi equation implies
\begin{equation*}
h_{ijk}=h_{ikj}.
\end{equation*}

The interchanging formula is given by 
\begin{equation}
h_{iijj}=\ h_{jjii}+h_{jj}h_{ii}^{2}-h_{jj}^{2}h_{ii}.  \label{comm}
\end{equation}

\begin{defi}\label{semi-convex}
We say a hypersurface $M$ is semi-convex if there exists a constant $K_0$
such that
\begin{equation*}
\kappa _{i}\left( X\right) \geq -K_0,\quad\forall 1\leq i\leq n,\quad\forall X\in M,
\end{equation*}
where $\kappa_i$'s are the principal curvatures of $M$.
\end{defi}

We will also use the following result by Lu \cite[Theorem 4.1]{LuSC}. 
\begin{theo}\label{sc thm}[Lu \cite{LuSC}] Let $n\geq 3$, $1\leq k\leq n$ and let $M$ be a semi-convex, strictly star-shaped hypersurface satisfying curvature equation (\ref{Equation-Sigma_k}) in $\mathbb{H}^{n+1}$ with $\kappa \in \Gamma _k$. Let $f\in C^{1,1}(\Gamma )$ be a positive function, where $\Gamma $ is an open neighborhood of the unit normal bundle of $M$ in $\mathbb{H}^{n+1}\times \mathbb{S}^{n}$. Then we
have 
\begin{equation*}
\max_{X\in M;1\leq i\leq n}|\kappa _{i}(X)|\leq C\left( 1+\max_{X\in\partial M;1\leq i\leq n}|\kappa _{i}(X)|\right) ,
\end{equation*}%
where $C$ is a constant depending only on $n, k, \Vert X\Vert _{C^{0,1}},\inf f$ and $\Vert f\Vert _{C^{1,1}}$.
\end{theo}

Note that above theorem works for $\mathbb{R}^{n+1}$ as well.

\section{A concavity inequality}

In this section, we will prove a concavity inequality for $\sigma_{n-1}$ operator when the smallest eigenvalue $\lambda_n$ is very negative. This is the key step towards the curvature estimate in the case that $M$ is not semi-convex.
\begin{lemm}\label{Lemma}
Let $n\geq 3$, let $F=\sigma_{n-1}$ and let $\lambda=(\lambda_1,\cdots,\lambda_n)\in \Gamma_{n-1}$ with 
\begin{align*}
\lambda _{1}=\cdots =\lambda _{m}>\lambda _{m+1}\geq \cdots \geq \lambda_{n}.
\end{align*}
Then there exists $K_{0}\geq 1$ depending only on $n$ and $\max F$ such that if $\lambda _{n}<-K_{0}$, then we have%
\begin{equation*}
-\sum_{p\neq q}F^{pp,qq}\xi _{p}\xi _{q}+\frac{\left(\sum_i F^{ii}\xi _{i}\right) ^{2}}{F}+2\sum_{i>m}\frac{F^{ii}\xi _{i}^{2}%
}{\lambda _{1}-\lambda _{i}}-\frac{F^{11}\xi _{1}^{2}}{\lambda
_1}\geq 0,
\end{equation*}
where $\xi=(\xi_1,\cdots,\xi_n)$ is an arbitrary vector in $\mathbb{R}^n$ satisfying $\xi_i=0$ for $1<i\leq m$.
\end{lemm}

\begin{proof}
Denote 
\begin{align*}
\Omega =-\sigma _{n}.
\end{align*}

By Lemma \ref{Sigma_k-lemma}, $\sigma_1(\lambda|12\cdots n-2)=\lambda_{n-1}+\lambda_n>0$. Thus $\lambda_i\geq |\lambda_n|$ for all $1\leq i\leq n-1$. Consequently,
\begin{equation}
\Omega=\lambda_1\cdots\lambda_{n-1}\cdot |\lambda_n| \geq \lambda _{1}\left( K_{0}\right) ^{n-1} . \label{omega}
\end{equation}

By Lemma \ref{Sigma_k-Lemma-0}, we have 
\begin{equation*}
\lambda _{i}F^{ii}=\lambda_i\sigma_{n-2}(\lambda|i)=\sigma_{n-1}-\sigma _{n-1}(\lambda |i)=F-\frac{\sigma _{n}}{\lambda
_{i}}.
\end{equation*}

Together with the definition of $\Omega$, we have
\begin{equation}\label{Con-1}
F^{ii}=\frac{F}{\lambda _{i}}+\frac{\Omega }{\lambda _{i}^{2}}.
\end{equation}

By Lemma \ref{Sigma_k-Lemma-0}, (\ref{Con-1}) and the definition of $\Omega$, for $j\neq i$, we have
\begin{align*}
\lambda _{j}F^{ii,jj}=&\ \lambda_j\sigma_{n-3}(\lambda|ij)=\sigma_{n-2}(\lambda|i)-\sigma _{n-2}(\lambda |ij)\\
=&\ F^{ii}-\frac{\sigma_n}{\lambda_i\lambda_j}=\frac{F}{\lambda _{i}}+\frac{\Omega }{\lambda _{i}^{2}}+\frac{\Omega }{\lambda _{i}\lambda _{j}}.
\end{align*}

It follows that 
\begin{equation}\label{Con-2}
F^{ii,jj}=\frac{F}{\lambda _{i}\lambda _{j}}+\frac{\Omega (\lambda
_{i}+\lambda _{j})}{\lambda _{i}^{2}\lambda _{j}^{2}}.
\end{equation}

By (\ref{Con-1}) and (\ref{Con-2}), we have
\begin{align}\label{Con-3}
 &\ -\sum_{p\neq q}F^{pp,qq}\xi _{p}\xi _{q}+\frac{\left(\sum_i F^{ii}\xi _{i}\right)^{2}}{F}\\\nonumber
 =&\ -\sum_{p\neq q}F^{pp,qq}\xi _{p}\xi _{q}+\sum_{p\neq q} \frac{F^{pp}F^{qq}\xi_p\xi_q}{F}+\sum_i \frac{(F^{ii})^2\xi_i^2}{F}\\\nonumber
 =&\ -\sum_{p\neq q}\left( \frac{F}{\lambda _{p}\lambda _{q}}+\frac{\Omega(\lambda _{p}+\lambda _{q})}{\lambda _{p}^{2}\lambda _{q}^{2}}\right) \xi _{p}\xi _{q}+\sum_{p\neq q}\left( \frac{F}{\lambda _{p}}+\frac{\Omega }{\lambda _{p}^{2}}\right) \left( \frac{F}{\lambda _{q}}+\frac{\Omega }{\lambda _{q}^{2}}\right) \frac{\xi _{p}\xi _{q}}{F} \\\nonumber
&\ +\sum_i \left( \frac{F}{\lambda _i}+\frac{\Omega }{\lambda _i^{2}}\right)^2\frac{\xi_i^2}{F}\\\nonumber
=&\ \sum_{p\neq q}\frac{\Omega^2 }{\lambda _{p}^{2}\lambda _{q}^{2}}\cdot \frac{\xi _{p}\xi _{q}}{F}+\sum_i \left( \frac{F^2}{\lambda _i^2}+2\frac{F\Omega}{\lambda_i^3}+\frac{\Omega^2 }{\lambda _i^4}\right)\frac{\xi_i^2}{F}\\\nonumber
=&\ \sum_{p,q}\frac{\Omega^2 }{F}\cdot\frac{\xi _{p}\xi _{q}}{\lambda_{p}^{2}\lambda _{q}^{2}}+\sum_i F\cdot\frac{\xi_i^2}{\lambda _i^2}+\sum_i 2\Omega\cdot \frac{\xi_i^2}{\lambda_i^3}.
\end{align}

By (\ref{Con-1}), we have
\begin{align}\label{Con-4}
&\ 2\sum_{i>m}\frac{F^{ii}\xi _{i}^{2}}{\lambda _{1}-\lambda _{i}}-\frac{F^{11}\xi _{1}^{2}}{\lambda _{1}}\\\nonumber
=&\ 2\sum_{i>m}\left(\frac{F}{\lambda _{i}}+\frac{\Omega }{\lambda _{i}^{2}}\right)\frac{\xi _{i}^{2}}{\lambda _{1}-\lambda _{i}}-\left(\frac{F}{\lambda _1}+\frac{\Omega }{\lambda _1^{2}}\right)\frac{\xi _{1}^{2}}{\lambda _{1}}\\\nonumber
=&\ F\left(2\sum_{i>m}\frac{\xi _{i}^{2}}{(\lambda _{1}-\lambda _{i})\lambda_i} -\frac{\xi_1^2}{\lambda_1^2}\right)+\Omega\left( 2\sum_{i>m}\frac{\xi _{i}^{2}}{(\lambda _{1}-\lambda _{i})\lambda _{i}^{2}}-\frac{\xi _{1}^{2}}{\lambda _{1}^3}\right).
\end{align}

Combining (\ref{Con-3}) and (\ref{Con-4}), we have
\begin{align}\label{Con-5}
& \ -\sum_{p\neq q}F^{pp,qq}\xi _{p}\xi _{q}+\frac{\left(\sum_i F^{ii}\xi _{i}\right)^{2}}{F}+2\sum_{i>m}\frac{F^{ii}\xi _{i}^{2}}{\lambda _{1}-\lambda _{i}}-\frac{F^{11}\xi _1^2}{\lambda _{1}}\\\nonumber
=&\ F\left( \sum_i\frac{\xi_i^{2}}{\lambda _i^{2}}+2\sum_{i>m}\frac{\xi _{i}^{2}}{(\lambda _{1}-\lambda_{i})\lambda _{i}}-\frac{\xi _{1}^{2}}{\lambda _{1}^2}\right)\\\nonumber
&\ +\Omega \left( \sum_{p,q}\frac{\Omega}{F}\cdot\frac{\xi _{p}\xi _{q}}{\lambda_{p}^{2}\lambda _{q}^{2}}+\sum_i\frac{2\xi _i^{2}}{\lambda _i^3}+2\sum_{i>m}\frac{\xi _{i}^{2}}{(\lambda _{1}-\lambda _{i})\lambda _{i}^{2}}-\frac{\xi _{1}^{2}}{\lambda _{1}^3}\right) \\\nonumber
=&\ F \cdot I+\Omega\cdot II.
\end{align}

Note that $\lambda_i>0$ for all $1\leq i\leq n-1$, thus
\begin{align*}
I=&\ \sum_{i\neq 1}\frac{\xi _i^{2}}{\lambda _i^{2}}+2\sum_{i>m}\frac{\xi _i^{2}}{(\lambda _{1}-\lambda _i)\lambda_i}\\
\geq&\  \left( \frac{1}{\lambda _{n}^{2}}+\frac{2}{(\lambda _{1}-\lambda_{n})\lambda _{n}}\right) \xi _{n}^{2}\\
=&\ \frac{\lambda_1+\lambda_n}{(\lambda_1-\lambda_n)\lambda _{n}^{2}}\cdot \xi _{n}^{2} \geq 0.
\end{align*}
We have used the fact that $\lambda_1+\lambda_n\geq 0$ in the last line by Lemma \ref{Sigma_k-lemma}.

Therefore, to prove the lemma, we only need to show $II\geq 0$. For the sake of convenience, define 
\begin{equation*}
\eta _{i}=\frac{\xi _{i}}{\lambda _{i}^{2}}.
\end{equation*}

By (\ref{Con-5}), we have
\begin{align}\label{matrix}
 II =&\ \sum_{p,q}\frac{\Omega }{F}\eta _{p}\eta_{q}+\sum_i2\lambda _i\eta _i^{2}+2\sum_{i>m}\frac{\lambda_i^{2}\eta _i^{2}}{\lambda _{1}-\lambda_i}-\lambda _{1}\eta _{1}^{2}\\\nonumber
=&\ \sum_{p,q}\frac{\Omega }{F}\eta _{p}\eta _{q}+\lambda _{1}\eta_{1}^{2}+\sum_{1<i\leq m}2\lambda _{i}\eta _{i}^{2}+\sum_{i>m}\frac{2\lambda_{1}\lambda _i}{\lambda _{1}-\lambda _i}\eta _i^{2}  \notag \\\nonumber
=&\ \sum_{p,q}a_{pq}\eta _{p}\eta _{q}.  
\end{align}

By assumption, $\xi_i=0$ for $1<i\leq m$, thus  $\eta_i=0$ for $1<i\leq m$. By deleting the rows and columns where $\eta _i=0$, we obtain an $\left( n-m+1\right) \times \left( n-m+1\right) $ submatrix of $\left( a_{pq}\right) $. This submatrix can be viewed as the sum of a rank $1$ matrix $s^{T}s$ and a diagonal matrix $D=diag(d_{1},d_{2},...,d_{n-m+1})$,
where 
\begin{align*}
s =&\ \sqrt{\frac{\Omega }{F}}\left[ 1,1,...,1\right], \\
d_{1}=&\ \lambda _{1},\quad d_{i-m+1}=\frac{2\lambda _{1}\lambda _i}{\lambda _{1}-\lambda _i},\quad\forall i>m.
\end{align*}

To prove $II\geq 0$, we only need to show $D+s^{T}s$ is positive definite. Note that
\begin{align}\label{Con-6}
\det \left( D+s^{T}s\right) =&\ \det \left( D\right) \det \left(I_{n-m+1}+D^{-1}s^{T}s\right) \\\nonumber
=&\ \det \left( D\right) \left(1+sD^{-1}s^{T}\right) \\\nonumber
=&\ \det \left( D\right) \left( 1+\frac{\Omega }{F}\sum_{k}\frac{1}{d_{k}}\right).
\end{align}

Since $\lambda_i=\lambda_1$ for $1<i\leq m$, we have
\begin{align*}
\sum_{k}\frac{1}{d_{k}}=&\ \frac{1}{\lambda _{1}}+\sum_{i>m}\frac{\lambda_{1}-\lambda _i}{2\lambda _{1}\lambda_i}=\ \frac{1}{\lambda _{1}}+\sum_i\frac{\lambda _{1}-\lambda _i}{2\lambda _{1}\lambda _i} \\
=&\ \frac{1}{\lambda _{1}}+\sum_i\frac{1}{2\lambda_i}-\frac{n}{2\lambda _{1}}=\ \left( 1-\frac{n}{2}\right) \frac{1}{\lambda _{1}}+\frac{\sigma_{n-1}}{2\sigma _{n}} \\
=&\ \left( 1-\frac{n}{2}\right) \frac{1}{\lambda _{1}}-\frac{F}{2\Omega }.
\end{align*}
We have used the fact that $\sum_i\frac{1}{\lambda_i}=\frac{\sigma_{n-1}}{\sigma_n}$ in the second line.

Plugging into (\ref{Con-6}), we have
\begin{align*}
\det \left( D+s^{T}s\right) = \det \left( D\right) \left( 1+\frac{\Omega }{F}\sum_{k}\frac{1}{d_{k}}\right)= \det \left( D\right) \left( \frac{1}{2}+\left( 1-\frac{n}{2}\right) \frac{\Omega }{\lambda _{1}F}\right).
\end{align*}

By (\ref{omega}), for $K_0$ sufficiently large, $\frac{1}{2}+\left( 1-\frac{n}{2}\right) \frac{\Omega }{\lambda _{1}F}<0$. On the other hand, $ \det \left( D\right) <0$ by definition of $D$. Thus $\det \left( D+s^{T}s\right)>0$.

Since $D$ has only one negative eigenvalue and $s^{T}s$ is nonnegative, together with the fact that  $\det \left( D+s^{T}s\right)>0$, we conclude that $D+s^{T}s$ is positive definite. Consequently $II\geq 0$. The proof of the lemma is now complete.
\end{proof}

\section{Curvature estimates}

In this section, we will prove Theorem \ref{Theorem-General}. It is a consequence of the following theorem.

\begin{theo}\label{theoremC2}
Let $n\geq 3$, $k=n-1$ and let $M$ be a strictly star-shaped hypersurface satisfying curvature equation (\ref{Equation-Sigma_k}) in $\mathbb{R}^{n+1}$ with $\kappa \in\Gamma_{n-1}$. Let $f\in C^{1,1}(\Gamma )$ be a positive function, where $\Gamma $ is an open neighborhood of the unit normal bundle of $M$ in $\mathbb{R}^{n+1}\times \mathbb{S}^{n}$. Then we have 
\begin{equation*}
\max_{X\in M;1\leq i\leq n}|\kappa _{i}(X)|\leq C\left( 1+\max_{X\in
\partial M;1\leq i\leq n}|\kappa _{i}(X)|\right) ,
\end{equation*}
where $C$ is a constant depending only on $n,\Vert X\Vert _{C^{0,1}},\inf f$ and $\Vert f\Vert _{C^{1,1}}$.
\end{theo}

\begin{proof}
Since $M$ is strictly star-shaped, without loss of generality, we may assume 
$u>a>0$. Consider the same test function as \cite{LuSC}
\begin{equation*}
Q=\ln \kappa _{1}-N\ln u+\frac{\alpha}{2}|X|^2 ,
\end{equation*}%
where $\kappa _{1}$ is the largest principle curvature and $N,\alpha $ are
large constants to be determined later. Assume $Q$ achieves maximum at an
interior point $X_{0}$. At $X_{0}$, we can choose an orthonormal frame such
that $(h_{ij})$ is diagonalized. Without loss of generality, we may assume $%
\kappa _{1}$ has multiplicity $m$, i.e. 
\begin{equation*}
\kappa _{1}=\cdots =\kappa _{m}>\kappa _{m+1}\geq \cdots \geq \kappa _{n}.
\end{equation*}

By Lemma 5 in \cite{BCD}, at $X_{0}$, we have 
\begin{equation}
\delta _{kl}\cdot {\kappa _{1}}_{i}=h_{kli},\quad 1\leq k,l\leq m,
\label{BCD}
\end{equation}%
\begin{equation*}
{\kappa _{1}}_{ii}\geq h_{11ii}+2\sum_{p>m}\frac{h_{1pi}^{2}}{\kappa
_{1}-\kappa _{p}},
\end{equation*}%
in the viscosity sense.

In the following, we will do a standard computation at $X_{0}$. The computation is exactly the same as in \cite{LuSC} up to (\ref{ineq}). For the sake of completeness, we include all details here. Readers who are familiar with this computation can jump to (\ref{ineq}).

At $X_0$, we have 
\begin{align}  \label{General-critical}
0=\frac{{\kappa_1}_i}{\kappa_1}-N\frac{u_i}{u}+\alpha\left\langle X,X_i\right\rangle=\frac{h_{11i}}{\kappa_1}-N\frac{u_i}{u}+\alpha\left\langle X,X_i\right\rangle,
\end{align}
\begin{align}  \label{General-1}
0\geq &\ \frac{{\kappa_1}_{ii}}{\kappa_1}-\frac{ (\kappa_1)_i^2}{\kappa_1^2}-N\frac{u_{ii}}{u}+N\frac{u_i^2}{u^2}+\alpha\left(g_{ii}-h_{ii}u\right)\\
\geq &\ \frac{h_{11ii}}{\kappa_1}+2\sum_{p>m}\frac{h_{1pi}^2}{%
\kappa_1(\kappa_1-\kappa_p)}-\frac{ h_{11i}^2}{\kappa_1^2}-N\frac{u_{ii}}{u}+\alpha\left(1-h_{ii}u\right),  \notag
\end{align}
in the viscosity sense.

By (\ref{comm}), we have 
\begin{equation*}
h_{11ii}=h_{ii11}+h_{11}^{2}h_{ii}-h_{ii}^{2}h_{11}.
\end{equation*}

Plugging into (\ref{General-1}), we have 
\begin{align*}
0\geq & \ \frac{h_{ii11}}{\kappa _{1}}+2\sum_{p>m}\frac{h_{1pi}^{2}}{\kappa
_{1}(\kappa _{1}-\kappa _{p})}-\frac{h_{11i}^{2}}{\kappa _{1}^{2}}-N\frac{%
u_{ii}}{u} \\
& \ +\alpha \left(1-h_{ii}u\right) +\kappa _{1}\kappa
_{i}-\kappa _{i}^{2}.
\end{align*}

Contracting with $F^{ii}=\sigma _{n-1}^{ii}$, together with Lemma \ref{Sigma_k-Lemma-0}, we have 
\begin{align}
0\geq & \ \sum_{i}\frac{F^{ii}h_{ii11}}{\kappa _{1}}+2\sum_{i}\sum_{p>m}\frac{F^{ii}h_{1pi}^{2}}{\kappa _{1}(\kappa _{1}-\kappa _{p})}-\sum_{i}\frac{F^{ii}h_{11i}^{2}}{\kappa _{1}^{2}}-N\sum_{i}\frac{F^{ii}u_{ii}}{u}
\label{General-2} \\
& \ +\alpha \sum_{i}F^{ii}-\alpha (n-1)Fu+(n-1)F\kappa_{1}-\sum_{i}F^{ii}\kappa _{i}^{2}  \notag \\
\geq & \ \sum_{i}\frac{F^{ii}h_{ii11}}{\kappa _{1}}+2\sum_{i}\sum_{p>m}\frac{F^{ii}h_{1pi}^{2}}{\kappa _{1}(\kappa _{1}-\kappa _{p})}-\sum_{i}\frac{F^{ii}h_{11i}^{2}}{\kappa _{1}^{2}}-N\sum_{i}\frac{F^{ii}u_{ii}}{u}  \notag
\\
& \ +\alpha\sum_{i}F^{ii}-\sum_{i}F^{ii}\kappa_{i}^{2}-C\alpha ,  \notag
\end{align}%
where $C$ is a universal constant depending only on $n,\Vert X\Vert_{C^{0,1}}$ and $\Vert f\Vert _{L^{\infty}}$. From now on, we will use $C$ to denote a universal constant depending only on $n,\Vert X\Vert_{C^{0,1}},\inf f$ and $\Vert f\Vert _{C^{1,1}}$, it may change from line to line.

By Lemma \ref{support function}, we have 
\begin{align*}
u_{ii}=\sum_kh_{iik}\left\langle X,e_k\right\rangle+h_{ii}-h_{ii}^2u.
\end{align*}

Together with Lemma \ref{Sigma_k-Lemma-0}, we have 
\begin{align*}
-N\sum_i\frac{F^{ii}u_{ii}}{u}=&\ -N\sum_k\frac{F_k\left\langle X,e_k\right\rangle}{u}-N\frac{(n-1)F}{u}+ N\sum_iF^{ii}\kappa_i^2 \\
\geq &\ -N\sum_k\frac{F_k\left\langle X,e_k\right\rangle}{u} + N\sum_iF^{ii}\kappa_i^2-CN.
\end{align*}

Plugging into (\ref{General-2}), we have 
\begin{align}
0\geq & \ \sum_{i}\frac{F^{ii}h_{ii11}}{\kappa _{1}}+2\sum_{i}\sum_{p>m}\frac{F^{ii}h_{1pi}^{2}}{\kappa _{1}(\kappa _{1}-\kappa _{p})}-\sum_{i}\frac{F^{ii}h_{11i}^{2}}{\kappa _{1}^{2}}-N\sum_{k}\frac{F_{k}\left\langle X,e_k\right\rangle}{u}
\label{General-3} \\
& \ +\alpha \sum_{i}F^{ii}+(N-1)\sum_{i}F^{ii}\kappa
_{i}^{2}-C\alpha -CN.  \notag
\end{align}

Differentiating equation (\ref{Equation-Sigma_k}) twice, we have 
\begin{align*}
&\ \sum_iF^{ii}h_{ii11}+\sum_{p,q,r,s}F^{pq,rs}h_{pq1}h_{rs1}=f_{11} \\
=&\ \left(\sum_kh_{k1}(d_\nu f)(e_k)+ (d_Xf)(X_1)\right)_1 \\
=&\ \sum_kh_{k11}(d_\nu f)(e_k)+h_{11}^2(d_{\nu\nu}
f)(e_1,e_1)+h_{11}(d_{\nu X} f)(e_1,X_1)-h_{11}^2(d_\nu f)(\nu) \\
&\ +(d_{XX}f)(X_1,X_1)+h_{11}(d_{X\nu}f)(X_1,e_1)-h_{11}(d_Xf)(\nu) \\
\geq &\ \sum_k h_{k11}(d_\nu f)(e_k)-C\kappa_1^2-C\kappa_1-C.
\end{align*}

Plugging into (\ref{General-3}), we have 
\begin{align}
0\geq & \ -\sum_{p,q,r,s}\frac{F^{pq,rs}h_{pq1}h_{rs1}}{\kappa _{1}}%
+2\sum_{i}\sum_{p>m}\frac{F^{ii}h_{1pi}^{2}}{\kappa _{1}(\kappa _{1}-\kappa
_{p})}-\sum_{i}\frac{F^{ii}h_{11i}^{2}}{\kappa _{1}^{2}}  \label{General-4}
\\
& \ -N\sum_{k}\frac{F_{k}\left\langle X,e_k\right\rangle}{u}+\alpha\sum_{i}F^{ii}+(N-1)\sum_{i}F^{ii}\kappa _{i}^{2}  \notag \\
& \ +\sum_{k}\frac{h_{11k}}{\kappa _{1}}(d_{\nu }f)(e_{k})-C\kappa
_{1}-C\alpha -CN.  \notag
\end{align}

By Lemma \ref{support function} and the critical equation (\ref%
{General-critical}), we have 
\begin{align*}
&\ -N\sum_k\frac{F_k\left\langle X,e_k\right\rangle}{u}+\sum_k\frac{h_{11k}}{\kappa_1}(d_\nu f)(e_k)
\\
=&\ -N\sum_k\bigg(h_{kk}(d_\nu f)(e_k)+ (d_Xf)(X_k)\bigg)\frac{\left\langle X,e_k\right\rangle}{u}+\sum_k\left( N\frac{u_k}{u}-\alpha\left\langle X,e_k\right\rangle\right) (d_\nu f)(e_k) \\
\geq &\ -N \sum_kh_{kk}(d_\nu f)(e_k)\frac{\left\langle X,e_k\right\rangle}{u}+N\sum_k\frac{h_{kk}\left\langle X,e_k\right\rangle}{u}(d_\nu f)(e_k)-CN-C\alpha \\
= &\ -CN-C\alpha.
\end{align*}

Plugging into (\ref{General-4}), we have 
\begin{align*}
0\geq & \ -\sum_{p,q,r,s}\frac{F^{pq,rs}h_{pq1}h_{rs1}}{\kappa _{1}}+2\sum_{i}\sum_{p>m}\frac{F^{ii}h_{1pi}^{2}}{\kappa _{1}(\kappa _{1}-\kappa_{p})}-\sum_{i}\frac{F^{ii}h_{11i}^{2}}{\kappa _{1}^{2}} \\
& \ +\alpha \sum_{i}F^{ii}+(N-1)\sum_{i}F^{ii}\kappa_{i}^{2}-C\kappa _{1}-C\alpha -CN.
\end{align*}

Now 
\begin{align*}
-\sum_{p,q,r,s} F^{pq,rs}h_{pq1}h_{rs1}=-\sum_{p\neq q}F^{pp,qq}h_{pp1}h_{qq1}+\sum_{p\neq q}F^{pp,qq}h_{pq1}^2.
\end{align*}

Thus 
\begin{align}
0\geq & \ -\sum_{p\neq q}\frac{F^{pp,qq}h_{pp1}h_{qq1}}{\kappa _{1}}+\sum_{p\neq q}\frac{F^{pp,qq}h_{pq1}^{2}}{\kappa _{1}}+2\sum_{i}\sum_{p>m}\frac{F^{ii}h_{1pi}^{2}}{\kappa _{1}(\kappa _{1}-\kappa _{p})}
\label{General-5} \\
& \ -\sum_{i}\frac{F^{ii}h_{11i}^{2}}{\kappa _{1}^{2}}+\alpha \sum_{i}F^{ii}+(N-1)\sum_{i}F^{ii}\kappa _{i}^{2}  \notag \\
& \ -C\kappa _{1}-C\alpha -CN.  \notag
\end{align}

By Lemma \ref{Sigma_k-Lemma-0}, we have
\begin{align}\label{General-6}
\sum_{p\neq q} \frac{F^{pp,qq}h_{pq1}^2}{\kappa_1}\geq 2\sum_{i>m}\frac{F^{11,ii}h_{11i}^2}{\kappa_1}= 2\sum_{i>m}\frac{(F^{ii}-F^{11})h_{11i}^2}{\kappa_1(\kappa_1-\kappa_i)}.
\end{align}

On the other hand
\begin{align}\label{General-7}
2\sum_i\sum_{p>m}\frac{F^{ii}h_{1pi}^2}{\kappa_1(\kappa_1-\kappa_p)}\geq&\ 2\sum_{p>m}\frac{F^{pp}h_{1pp}^2}{\kappa_1(\kappa_1-\kappa_p)}+2\sum_{p>m}\frac{F^{11}h_{1p1}^2}{\kappa_1(\kappa_1-\kappa_p)}.
\end{align}

Combining (\ref{General-6}) and (\ref{General-7}), we have
\begin{align}\label{General-8}
\sum_{p\neq q} \frac{F^{pp,qq}h_{pq1}^2}{\kappa_1}+2\sum_i\sum_{p>m}\frac{F^{ii}h_{1pi}^2}{\kappa_1(\kappa_1-\kappa_p)}\geq 2\sum_{i>m}\frac{F^{ii}h_{ii1}^2}{\kappa_1(\kappa_1-\kappa_i)}+2\sum_{i>m}\frac{F^{ii}h_{11i}^2}{\kappa_1(\kappa_1-\kappa_i)}.
\end{align}

By (\ref{BCD}), we have 
\begin{align}  \label{General-9}
h_{11i}=h_{1i1}=\delta_{1i}\cdot {\kappa_1}_1=0,\quad \forall 1<i\leq m.
\end{align}

Plugging (\ref{General-8}) and (\ref{General-9}) into (\ref{General-5}), we have
\begin{align*}
0\geq &\ -\sum_{p\neq q}\frac{F^{pp,qq}h_{pp1}h_{qq1}}{\kappa_1}+2\sum_{i>m}\frac{F^{ii}h_{ii1}^2}{\kappa_1(\kappa_1-\kappa_i)}+\sum_{i>m}\frac{F^{ii}(\kappa_1+\kappa_i)h_{11i}^2}{\kappa_1^2(\kappa_1-\kappa_i)}-\frac{ F^{11}h_{111}^2}{\kappa_1^2}\\
&\ +\alpha \sum_iF^{ii}+(N-1)\sum_iF^{ii}\kappa_i^2-C\kappa_1-C\alpha-CN.
\end{align*}

By Lemma \ref{Sigma_k-lemma}, $\lambda_1+\lambda_i\geq 0$ for all $1<i\leq n$. Consequently,
\begin{align}\label{ineq}
0\geq&\ -\sum_{p\neq q}\frac{F^{pp,qq}h_{pp1}h_{qq1}}{\kappa_1}+2\sum_{i>m}\frac{F^{ii}h_{ii1}^2}{\kappa_1(\kappa_1-\kappa_i)}-\frac{ F^{11}h_{111}^2}{\kappa_1^2}\\\nonumber
&\ +\alpha \sum_iF^{ii}+(N-1)\sum_iF^{ii}\kappa_i^2-C\kappa_1-C\alpha-CN.
\end{align}

Let $K_{0}$ be the constant in Lemma \ref{Lemma}. We now separate into two cases. 

\medskip

{\bf Case i:} At $X_{0}$, $\kappa _{n}\geq -K_{0}$.

In this case, the hypersurface will become semi-convex at $X_{0}$. Since we use the same test function as Lu \cite[Theorem 4.1]{LuSC},
Theorem \ref{theoremC2} will hold automatically by Theorem \ref{sc thm}. We
remark that the proof for Theorem \ref{sc thm} only requires the
hypersurface to be semi-convex at the maximum point of the test function.

\medskip

{\bf Case ii:}  At $X_{0}$, $\kappa _{n}<-K_{0}$.

In this case, we first note that by (\ref{BCD}),
\begin{align*}
h_{ii1}=h_{i1i}=\delta_{i1}\cdot {\kappa_1}_i=0,\quad \forall 1<i\leq m.
\end{align*}

Apply Lemma \ref{Lemma} to equation (\ref{ineq}), we obtain 
\begin{align*}
0 \geq &\ -\frac{\left(\sum_i F^{ii}h_{ii1}\right) ^{2}}{F\kappa _{1}}+\alpha \sum_{i}F^{ii}+(N-1)\sum_{i}F^{ii}\kappa _{i}^{2}\ -C\kappa
_{1}-C\alpha -CN \\
\geq &\ (N-1)\sum_{i}F^{ii}\kappa_{i}^{2}\ -C\kappa _{1}-C\alpha -CN,
\end{align*}
where we have used the equation $\sum_i F^{ii}h_{ii1}= h_{11}(d_\nu f)(e_1)+ (d_Xf)(X_1)$ in the second line.

By Lemma \ref{Sigma_k-lemma}, we have $F^{11}\kappa _{1}^{2}\geq c(n)\kappa _{1}$. It follows that
\begin{align*}
0\geq \left(c(n)N-C\right)\kappa_1-C\alpha-CN.
\end{align*}

By choosing $N$ sufficiently large, we conclude that $\kappa_1\leq C$.

We remark that in case (i), $\alpha $ and $N$ also have a lower bound. The
large constants we fix have to satisfy both case (i) and case (ii).
\end{proof}

\end{document}